\documentclass{amsart}
\usepackage{graphicx, xcolor}

\usepackage{mathtools,
amsthm,
enumitem,
tikz-cd,
microtype,
leftindex}

\usepackage{cellspace} 
\usepackage{hyperref}
\hypersetup{
    colorlinks,
    linkcolor={red!50!black},
    citecolor={blue!50!black},
    urlcolor={blue!80!black}
}

\newtheorem{theorem}{Theorem}
\newtheorem{proposition}[theorem]{Proposition}
\newtheorem{lemma}[theorem]{Lemma}

\theoremstyle{definition}
\newtheorem{definition}[theorem]{Definition}
\newtheorem{example}[theorem]{Example}
\newtheorem{construction}[theorem]{Construction}
\newtheorem{remark}[theorem]{Remark}
\newtheorem{algorithm}[theorem]{Algorithm}

\newcommand{\lowerstar}{\mathbin{\underline\star}}
\newcommand{\upperstar}{\mathbin{\overline\star}}
\newcommand{\im}{\operatorname{im}}
\newcommand{\Mixt}{\operatorname{Mixt}}

\title[Markov Combinations of Discrete Statistical Models]{Markov Combinations of\\Discrete Statistical Models}
\author{Orlando Marigliano, Eva Riccomagno}
\date{\today}

\begin{document}
\begin{abstract}
Markov combination is an operation that takes two statistical models and produces a third whose marginal distributions include those of the original models. Building upon and extending existing work in the Gaussian case, we develop Markov combinations for categorical variables and their statistical models. We present several variants of this operation, both algorithmically and from a sampling perspective, and discuss relevant examples and theoretical properties. We describe Markov combinations for special models such as regular exponential families, discrete copulas, and staged trees. Finally, we offer results about model invariance and the maximum likelihood estimation of Markov combinations.
\end{abstract}

\maketitle

\section{Introduction}

This paper describes a method for combining two statistical models to obtain a third, a concept with
both theoretical and practical implications.
From a theoretical point of view, such an operation can be analysed for its algebraic properties such as associativity, commutativity, and unit elements. In addition,  statistical properties of the original models may transfer to their combination. Known models may be representable as a sequence of combinations of simpler models, and the combination of models of the same type may result in a model of that type.

From a data analysis perspective, this model-based approach offers an alternative to traditional methods for merging evidence from different sources such as meta-analysis, data fusion and data integration. Combining  two statistical models for different datasets produces a model for the combined dataset in a way that conforms to the original models, as detailed in Section~\ref{sec:definitions}.  There it is shown that sampling from the combined model entails first sampling from one model and then from the second model within the same meta-category of the first sample. Different combination variants provide variations of this basic structure.

The operation, known as \emph{Markov combination}, acts on probability density functions. It was first introduced in~\cite{dawid-lauritzen} to join two graphical
models along a common clique and later generalised in~\cite{massa-lauritzen} and~\cite{massa-riccomagno} for a finite set of Gaussian random variables. Markov combinations also form the basis for the notion of Markov melding~\cite{markov-melding}. In this paper the concept is focused on the discrete categorical case, to allow for the combination of discrete models.

The basic setup involves a discrete random variable defined on a finite set $I$, which represents the set of possible values, categories, or levels that the variable can take. The case of multiple random variables can be reduced to this form by mapping the level sets to a suitable $I$. The order of the elements in $I$ is not relevant for our work.  The overall setup is similar to that used in~\cite{relational-models},~\cite{pistone-riccomagno}, and~\cite{diaconis-sturmfels}.

The paper is structured as follows.

Section~\ref{sec:definitions} defines the operation of meta-categorisation  on 
a finite set $I$, the basic operation of Markov combination of statistical models, and several variants of the latter. These operations are  motivated and interpreted in terms of sampling. Section~\ref{sec:examples} discusses examples and special cases and shows that discrete Markov combinations generalise both independent Cartesian products and the Markov combinations defined in~\cite{dawid-lauritzen}. Section~\ref{sec:properties} examines the theoretical properties of Markov combinations, focusing on associativity and the formulation of mixture models as chains of combinations and marginals.

Section~\ref{sec:models} describes Markov combinations for specific  classes of models, namely discrete copulas~\cite{discrete-copulas}, regular exponential families (saturated models), and staged trees~\cite{goergen-bigatti}. Section~\ref{sec:theorems} presents theoretical results,
including model invariance and maximum likelihood estimation for Markov combinations of saturated models. Section~\ref{sec:algorithms} provides algorithms for computing density functions and sampling. Section~\ref{sec:discussion} concludes the paper with a discussion of future directions.

\section{Definitions}\label{sec:definitions}

When analysing categorical data, it is often useful to aggregate,  collapse, or recode groups of data categories into broader categories. After introducing a formalism for this operation of meta-categorisation, we use it to define Markov combinations of discrete statistical models for two separate categorical variables using appropriately-chosen metacategories.

\subsection{Category mappings and aggregates}

A \emph{meta-categorisation} is the grouping of categories from a finite set $I$ into broader categories $M$, called \emph{metacategories}. This can be formalised in two ways: as a partition of $I$ or as a surjective map $p : I \to M$. The second approach, which we call \emph{category mapping}, is preferable because it explicitly names the metacategories through the set $M$. 

\begin{definition}
Let $I$ be a finite set. A \emph{category mapping} of $I$ is a finite set $M$ together with a surjective function $p:I\to M$.
The corresponding \emph{category aggregation} is the partition on $I$ induced by $p$. Elements of $M$ are called \emph{metacategories}. For a metacategory $k \in M$, the
\emph{aggregate category} $I_k$ is the subset $\{ i \in I \mid p(i) = k \}$.
\end{definition}

The categories often index count data, relative frequencies, probability distributions, and statistical models of discrete random variables. In each case, we can aggregate data by summing over the elements $i\in I$ mapped to each particular metacategory. Because of this broad
applicability, we shall give the next definition for $I$-indexed vectors of any given field.

\begin{definition}\label{aggregates}
Let $K$ be a field, $p:I\to M$ a category mapping and $ f = (f_i)_{i\in I} \in K^I$ an $I$-indexed vector with entries in $K$. The \emph{aggregate} of $f$ with respect to $p$ is the vector $f_M\in K^M$ 
whose $k$-th component is given by
\[
f_{M,k} = \sum_{i\in I_k} f_i \quad
\text{for all } k\in M.
\]
\end{definition}

\begin{example} [Parametric probability density functions]
    A discrete \emph{parametric statistical model} over the set of categories $I$ is given by an $I$-indexed sequence of functions $f_i:\Theta\to \mathbb R_{\geq 0}$ such that $\sum_{i\in I}f_i$ is the constant function $1$. Each $f_i$ represents the probability for a random sample to fall into the category $i$, parametrised by $\theta\in\Theta$. To apply Definition~\ref{aggregates} to statistical models we take $K$ to be the field of ratios of functions $\Theta \to \mathbb R$.
    Then, the aggregate $f_M$ of a parametric statistical model $f$ is a parametric statistical model on $M$ whose components are linear combinations of coordinates of $f$. Different families of statistical models can be defined by restricting $K$. For example, in Algebraic Statistics, $K$ is often the field of (polynomial) rational functions in $\Theta$~\cite{pistone-riccomagno,drton-sturmfels}.
\end{example}

\begin{example} [Marginal distributions] \label{marginal}
Marginalisation is a special type of aggregate. Indeed, if $A$ and $B$ are sets of categories for discrete random variables $X$ and $Y$, then $I = A\times B$ is the set of categories for the random vector $(X,Y)$. Let $M=A$ and let $p:I\to M$ be the projection on the first factor $p(a,b) = a$. If $f$ is a probability density on $I$, then its aggregate with respect to $p$ is the usual marginal density function $f_X$ of $f$ with respect to the random variable $X$. Indeed, for $a\in A$ we have $I_a=\{(a,b)\mid b\in B\}$ and
\[
f_{M, a} = \sum_{i\in I_a} f_i = \sum_{b\in B} f(a,b) = f_{X,a}.
\]
\end{example}

\subsection{Mapping products}
When combining two categorical variables $X$ and $Y$, we typically consider the Cartesian product $I \times J$ as the set of
possible joint categories. However, if $I$ and $J$ have a shared Cartesian structure, such as $I = I_0 \times I_1$ and $J = I_1 \times
I_2$, a simpler combined set $I_0 \times I_1 \times I_2$ may be preferred to $I_0 \times I_1 \times I_1 \times I_2$. To enable
such flexibility, we specify a common set of metacategories $M$ for $I$ and $J$, and only include pairs $(i, j) \in I \times J$
where $p(i) = q(j)$, where $p$ and $q$ are category mappings into $M$.

\begin{definition}\label{mapping-product}
For finite sets $I$ and $J$,
let $p:I\to M$ and $q:J\to M$ be category mappings. The \emph{mapping product} (or \emph{pullback} or \emph{fiber product}) 
of $p$ and $q$ is the set $I\times_M J \coloneqq\{(i,j)\in I\times J \mid p(i) = q(j)\}$.
\end{definition}

Here, we have modeled the mapping product as a subset of the Cartesian product of $I$ and $J$. In particular, every distribution on $I\times_M J$ can be extended to $I\times J$ by setting the missing values to zero. 
The mapping product can also be represented as the disjoint union of Cartesian products over aggregate categories:
$
I \times_M J = \coprod_{k\in M} I_k\times J_k$, where $\coprod$ stands for the disjoint union of sets. 
The cardinality of the mapping product is $\sum_{k \in M} |I_k| \cdot |J_k|$. The usual Cartesian product
is a special case of the mapping product where $M$ has only one element and all categories belong to the same
metacategory.

\subsection{Markov combinations}

Suppose we have two sets of categories $I$ and $J$ with the same metacategories, a set of samples $\mathcal A$ labeled by $I$, and a set of samples $\mathcal B$ labeled by $J$, thus we view $\mathcal A$ as a multiset of $I$ and $\mathcal B$ as a multiset of $J$. 
We can construct a combined set of samples by the following procedure: first, sample an element from $\mathcal A$; then, conditional
on the metacategory of this sample, draw a sample from $\mathcal B$, i.e.\ within the same metacategory. Alternatively, we may reverse the order: first sampling from $\mathcal B$, then from $\mathcal A$ within the same metacategory. This motivates the following definition, which describes how such data---represented as frequency distributions over the categories---combine. In both variants, the combined set of samples is labeled by the mapping product $I\times_M J$.

\begin{definition}
Let $K$ be a field, and let $p:I\to M$ and $q:J\to M$ be category mappings. The \emph{left Markov combination} of two vectors $f\in K^{I}$ and $g\in K^{J}$ is the vector $f\ast_M g\in K^{I\times_M J}$ such that
\[
(f\ast_M g)_{i,j}= \frac{f_i \ g_j}{f_{M,p(i)}} \quad
\text{for all } (i,j)\in I\times_M J. 
\]
Analogously, the \emph{right Markov combination} of $f$ and $g$ is defined as the vector $f {\leftindex_M {\ast }\, }g \in K^{I\times_M J}$ given by $(f {\leftindex_M {\ast }\, } g )_{i,j}\coloneqq f_i \ g_j/g_{M,q(j)}$ for all $(i,j)\in I\times_M J$.
\end{definition}

Later we will use the fact that $f {\leftindex_M {\ast }\, } g = g \ast_M f$.
Like aggregates, Markov combinations can be defined over different fields $K$, which enables application at different levels of analysis. For instance, let $f = (\alpha_i)_{i \in I}$ and $g = (\beta_j)_{j \in J}$ be probability density functions, i.e., $\alpha_i, \beta_j \in [0,1]$ with $\sum_i \alpha_i = \sum_j \beta_j = 1$. Then $f \ast_M g$ defines a density function on $I \times_M J$.
One can also combine two parametric statistical models $(f_i)_{i \in I}$ and $(g_j)_{j \in J}$, where $f_i, g_j : \Theta \to \mathbb{R}_{\geq 0}$ for all $i$ and $j$, by operating within the field of fractions of the ring of real-valued functions on $\Theta$. The result is again a statistical model.

By using the basic operations of left and right Markov combinations of statistical models, we can define several variants, each characterised by a different interpretation in terms of sampling. To this end, we adapt the core model combinations listed in~\cite{dawid-lauritzen} and~\cite{massa-lauritzen} to our setup for discrete models.
To formalise this idea, we consider Markov combinations within the function field $K(\Theta)$, i.e.\ the field of ratios of functions $\Theta \to \mathbb{R}$. A statistical model is then represented as a vector in $K(\Theta)^I$, i.e., as a vector of rational functions indexed by $I$.

\begin{definition}\label{markov-variants}
Let $I$ and $J$ be finite sets and $p:I\to M$ and $q:J\to M$  category mappings.
For a field $K$ and a parameter space $\Theta$, let $K(\Theta)$ be the field of ratios of functions $\Theta\to \mathbb R$. 
\begin{enumerate}

\item  Two vectors $f\in K^I$ and $g\in K^J$ are said to be \emph{consistent} if their aggregates w.r.t.\ $p$ and $q$ coincide, i.e., if $f_M = g_M$. In this case, the left and right Markov combinations coincide and the \emph{Markov combination} of $f$ and $g$ is defined as $f\star g =f\ast_M g=f {\leftindex_M {\ast }\, }g$.

\item Two parametric statistical models $f\in K(\Theta)^I$ and $g\in K(\Theta)^J$ are said to be \emph{meta-consistent} if for every $\theta \in \Theta$, the evaluations $f(\theta)$ and $g(\theta)$ are consistent. In this case, we define their \emph{meta-Markov combination} as the 
vector $f \star g\in K(\Theta)^{I\times_M J}$  obtained by taking the Markov combination pointwise in $\theta$.

\item Let $f\in K(\Theta)^I$ and $g\in K(\Theta)^J$ be parametric models and define the parameter space $\Theta'\coloneqq\{(\theta_1,\theta_2)\in\Theta\times \Theta\mid f(\theta_1), g(\theta_2) \text{ consistent}\}$. 
    \begin{enumerate} 

    \item The \emph{lower Markov combination} of $f$ and $g$ is the vector $f \lowerstar g$ in $K(\Theta')^{I\times_M J}$ defined for all $(\theta_1,\theta_2)\in\Theta'$ by
    \begin{equation*}
    (f\lowerstar g)(\theta_1,\theta_2) = f(\theta_1)\star g(\theta_2).
    \end{equation*}

    \item The \emph{upper Markov combination} of $f$ and $g$ is the vector $f\upperstar g$ in $K(\Theta\times \Theta\times \{0,1\})^{I\times_M J}$ defined for all $\theta_1,\theta_2\in\Theta$ by 
\begin{align*}
(f\upperstar g)(\theta_1, \theta_2, 0) &= f(\theta_1)\ast_M g(\theta_2),\\
(f\upperstar g)(\theta_1, \theta_2, 1) &= f(\theta_1) {\leftindex_M {\ast }\, } g(\theta_2).
\end{align*}

    \item The \emph{super Markov combination} of $f$ and $g$ is the vector $f\otimes g$ in $K(\Theta^3\times \{0,1\})^{I\times_M J}$ defined for all $\theta_1,\theta_2,\theta_3\in \Theta$, $(i,j)\in I\times_M J$ by 
\begin{align*} 
(f\otimes g)_{i,j}(\theta_1,\theta_2,\theta_3,0) 
&=\frac{f_i(\theta_1)}{f_{M,p(i)}(\theta_1)}\cdot f_{M,p(i)}(\theta_2) \cdot\frac{g_j(\theta_3)}{g_{M,q(j)}(\theta_3)},\\
(f\otimes g)_{i,j}(\theta_1,\theta_2,\theta_3,1) 
&=\frac{f_i(\theta_1)}{f_{M,p(i)}(\theta_1)}\cdot g_{M,q(j)}(\theta_2) \cdot\frac{g_j(\theta_3)}{g_{M,q(j)}(\theta_3)}.
\end{align*}  
    \end{enumerate}

\end{enumerate}
\end{definition}

\begin{example}
    Let $I=\{ a,b,c\}$,  $J=\{ \clubsuit, \diamondsuit,
\heartsuit, \spadesuit \}$, $M=\{1,2\}$, and define the two category mappings $p$ and $q$ by
\[
\begin{array}{lll}
p(a) =1 &&  p(b)=p(c)=2 
\\
q(\clubsuit)=q( \diamondsuit) =q(\heartsuit)=1 && q(\spadesuit)=2. 
\end{array}
\]
The two density vectors $f =(\frac34, \frac18, \frac18)$ defined on $I$ and $g=(\frac14,\frac14,\frac14,\frac14)$ defined on $J$ are consistent w.r.t. the given category mappings, while neither is consistent with the vector $(\frac18,\frac18,\frac34)$ defined on $I$.
Similarly, for $\theta\in[0,\frac14]$ the two statistical models 
$f =(3\theta, \theta, 1-4\theta)$ on $I$ and $g=(\theta,\theta,\theta,1-3\theta)$ on $J$ are consistent, while $f=(1-4\theta, 2\theta,2\theta)$ on $I$ and $(\theta,\theta,\theta,1-3\theta)$ on $J$ are consistent only for $\theta=1/7$.
\end{example}

Note that a) all notions of Markov combination depend on the choice of category mappings to $M$ which we leave 
implicit in the terminology,
b) the operator $\star$ is commutative, and
c) one could write upper Markov combinations using only left Markov combinations since
$(f\upperstar g)(\theta_1, \theta_2, 1) = f(\theta_1) {\leftindex_M {\ast }\, } g(\theta_2) =
g(\theta_2) {\ast_M } f(\theta_1)$. 
 
The distinction between left and right Markov combinations reflects the lack of commutativity of these operations for non-consistent densities. This motivates the separation into cases indexed by $0$ and $1$ in Definition~\ref{markov-variants}(3b-3c). 
From the meta-Markov combination of two meta-consistent models, the original models can be recovered by taking aggregates, see Example~\ref{aggregate-meta-markov} for more details. 

The combinations defined in Definition~\ref{markov-variants} can be interpreted in terms of sampling schemes. Let $\mathcal{A}$ and $\mathcal{B}$ be two multisets of $I$ and $J$, respectively.

- For the basic Markov combination in Definition~\ref{markov-variants}(1), we can sample from $\mathcal{A}$ and then from $\mathcal{B}$ within the same metacategory, or vice versa. 
This corresponds to taking the left or the right Markov combination, which coincide if the factor distributions are consistent.

- For the meta-Markov combination in Definition~\ref{markov-variants}(2), we assume that $\mathcal{A}$ and $\mathcal{B}$ are generated by parametric models. Fixing a parameter $\theta \in \Theta$, we perform sampling from $\mathcal{A}$ and then from $\mathcal{B}$ within the same metacategory, or vice versa. If, for all $\theta$, the two distributions obtained by sampling first from $\mathcal A$ or first from $\mathcal B$  are equal, then the models are meta-consistent. 

- Definition~\ref{markov-variants}(3) addresses the case of not meta-consistent models. For the lower Markov combination in Definition~\ref{markov-variants}(3a) the model parameters $\theta_1$ and $\theta_2$ may differ between sampling from $\mathcal A$ or from $\mathcal B$, as long as they yield consistent distributions.
For the upper Markov combination in Definition~\ref{markov-variants}(3b), the binary parameter $0$ or $1$ determines whether a left or right combination is taken. Here, a change of parameter $\theta$ between sampling from $\mathcal A$ and from $\mathcal B$ is allowed.

- For the super Markov combination in Definition~\ref{markov-variants}(3c), we allow full flexibility: parameters may vary between all stages, and the choice of left or right combination is encoded explicitly. We first draw a sample to fix a metacategory, then we sample from both $\mathcal A$ and $\mathcal B$ whithin that metacategory. The fourth parameter decides whether the metacategory is decided by sampling from $\mathcal A$ or from $\mathcal B$.

In practical applications, one might wish to restrict the flexibility of parameter variation across sampling stages. The following definition introduces the notion of \emph{restricted Markov combinations} to capture such scenarios.

\begin{definition}\label{restricted-markov-variants}
Let $f\in K(\Theta)^I$ and $g\in K(\Theta)^J$ be parametric models.
\begin{enumerate}
\item[($3'$a)] The \emph{restricted lower Markov combination} of $f$ and $g$ is the vector $f\lowerstar g$ in $K(\Theta')^{I\times_M J}$ where $\Theta'\coloneqq\{\theta\in \Theta\mid f(\theta), g(\theta) \text{ consistent}\}$ and
\[
(f\lowerstar g)(\theta) = f(\theta)\star g(\theta)
\quad \text{for all } \theta\in \Theta'.
\]
\item[($3'$b)] The \emph{restricted upper Markov combination} of $f$ and $g$ is the vector $f\upperstar g$ in $K(\Theta\times\{0,1\})^{I\times_M J}$ where
\begin{align*}
(f\upperstar g)(\theta, 0) &= f(\theta)\ast_M g(\theta),\\
(f\upperstar g)(\theta, 1) &= f(\theta) {\leftindex_M {\ast }\, } g(\theta)
\end{align*}
for all $\theta\in \Theta$.
\item[($3'$c)] The \emph{restricted super Markov combination} of $f$ and $g$ is the vector $f\otimes g$ in $K(\Theta^2 \times\{0,1\})^{I\times_M J}$ where
\begin{align*} 
(f\otimes g)_{i,j}(\theta_1,\theta_2,0) 
&=\frac{f_i(\theta_1)}{f_{M,p(i)}(\theta_1)}\cdot f_{M,p(i)}(\theta_2) \cdot\frac{g_j(\theta_1)}{g_{M,q(j)}(\theta_1)},\\
(f\otimes g)_{i,j}(\theta_1,\theta_2,1) 
&=\frac{f_i(\theta_1)}{f_{M,p(i)}(\theta_1)}\cdot g_{M,q(j)}(\theta_2) \cdot\frac{g_j(\theta_1)}{g_{M,q(j)}(\theta_1)}
\end{align*}
for all $\theta_1,\theta_2\in \Theta$ and $(i,j)\in I\times_M J$.
\end{enumerate}
\end{definition}

\subsection{Structured Super-Markov}

We present a slight generalisation of the Super-Markov combinations of Definition~\ref{markov-variants}(3c). Two 
category mappings $I\to M$ and $J\to M$ and three statistical models are assumed: $f(\theta_1)$ over $I$, $h(\theta_2)$ over $M$ and $g(\theta_3)$ over $J$. The \emph{structured Super-Markov combination} of $f$ and $g$ with respect to $h$ is the model over $I\times_M J$ which is parametrised by $(\theta_1,\theta_2,\theta_3)$ and for $k\in M$, $i\in I_k$ and $j\in J_k$ takes the value
\[
(f\otimes_h g)_{i,j}(\theta_1,\theta_2,\theta_3) \coloneqq 
\frac{f_i(\theta_1)}{f_{M,k}(\theta_1)}\cdot h_k(\theta_2)\cdot \frac{g_j(\theta_3)}{g_{M,k}(\theta_3)}.
\]

We can recover the 
Super-Markov combination by choosing $h(\theta_2)$ in such a way that all aggregates of $f$ and $g$ can occur.
An example of this can be found
in Definition~\ref{markov-variants} where we 
set
$\theta_2 = (\theta_2',\alpha)$, 
with $\alpha\in\{0,1\}$ 
indicating
whether to take $h=f_M(\theta_2')$ or $h=g_M(\theta_2')$.

We calculate the aggregates of the Super-Markov as follows:
\[
(f\otimes_h g)_{I,i}(\theta_1,\theta_2,\theta_3)
=
\frac{f_i(\theta_1)}{f_{M,k}(\theta_1)}h_k(\theta_2)
\sum_{j\in J_k}
\frac{g_j(\theta_3)}{g_{M,k}(\theta_3)}
=
\frac{f_i(\theta_1)}{f_{M,k}(\theta_1)}\cdot h_k(\theta_2)
\]
since the summation equals one.
This expression does not depend on $\theta_3$ and we may write it as $(f\otimes_h g)(\theta_1, \theta_2)$. Analogously, the other aggregates over $J$ and $M$ are
\[
(f\otimes_h g)_{J,j}(\theta_2,\theta_3) = h_k(\theta_2)\cdot \frac{g_j(\theta_3)}{g_{M,k}(\theta_3)}\quad\text{and} \quad
(f\otimes_h g)_{M,k}(\theta_2)=h_k(\theta_2)
\]
Note that the two aggregates $(f\otimes_h g)_I$ and $(f\otimes_h g)_J$ are meta-consistent.

A notable property of the structured Super-Markov, and therefore of the ordinary Super-Markov, is that it is closed with respect to repeated application and marginalisation, as follows.
\begin{description}
\item[In the meta-consistent case we have] 
$
(f\star g)_I \star (f\star g)_J = f\star g
$
(see Remark~\ref{aggregate-meta-markov}.)
\item[In the structured Super-Markov case we have]
\begin{align*}
&
\bigl((f\otimes_h g)_I\otimes_h (f\otimes_h g)_J\bigr)_{i,j}\bigl(
(\theta_1,\theta_2), \theta'_2, (\theta_2'',\theta_3)
\bigr)
\\ &= \frac{(f\otimes_h g)_{I,i}(\theta_1,\theta_2)}{(f\otimes_h g)_{M,k}(\theta_2)}\cdot h_k(\theta_2')\cdot \frac{(f\otimes_h g)_{J,j}(\theta_2'',\theta_3)}{(f\otimes_h g)_{M,k}(\theta''_2)}\\
&=
\frac{f_i(\theta_1)}{f_{M,k}(\theta_1)}\cdot h_k(\theta_2')\cdot \frac{g_j(\theta_3)}{g_{M,k}(\theta_3)}\\
&= (f\otimes_h g)_{i,j}(\theta_1,\theta_2',\theta_3).
\end{align*}

\end{description}
This confirms and generalises the result in~\cite{massa-lauritzen} by which the Super-Markov combination is stable under structured recombination of its marginal components.

\section{Examples}\label{sec:examples}

\begin{example}[Structural zeros and independence]
Suppose we want to model the academic career of a PhD student. We record two variables which we assume to be independent: the number of years $X$ she spends in her doctoral program and the number of publications $Y$ she 
produces during that time.
To keep the state space finite, we set upper limits: $x_0 = 5$ for the number of years and $y_0 = 10$ for the number of publications. Values beyond these thresholds are treated as censored and we only record that $X \geq x_0$ or $Y \geq y_0$. 
 
To model this in our setup, we take $I=\{0,\dotsc,x_0\}$, $J=\{0,\dotsc,y_0\}$, $M = \{0,1\}$ and category mappings $p$ on $I$ and $q$ on $J$ 
such that $p(0)=q(0)=0$ and all other values of $I$ and $J$ are mapped into $1\in M$. Finally we take the mapping product 
\[I\times_M J = \{(0,0), (x,y)\mid x\in I, y\in J, x,y>0\}\]
as the state space for the joint random vector $Z = (X, Y)$. 
This reflects two modelling assumptions:
\begin{enumerate}
\item if the doctoral programme was completed in any number of years ($x > 0$), then there is at least one publication (e.g. the dissertation itself), i.e. $y > 0$.
\item If there is at least one publication ($y > 0$), then the doctoral programme must have been running for some time, i.e. $x > 0$.
\end{enumerate}
Therefore, $(0,0)$ is the only valid case in which one of the two variables is zero.

Let $f_X$ be a probability law for $X$ given as a density function on $I$, and analogously define $g_Y$.
The independence assumptions would give the joint distribution $f(x,y) = f_X(x) \ g_Y(y)$ for $(x,y) \in I\times J$. But the joint model is defined only on 
$ I\times_M J$ and the  modelling assumptions imply $f_X(0) = g_Y(0)$, that is $f_X$ and $g_Y$ are consistent w.r.t.~$M$.
If $f_X$ and $g_Y$ are from parametric families, e.g.\ both are Poisson distributed with unknown parameters, then consistency imposes equality of parameters and $f_X$ and $g_Y$ are meta-consistent.
Finally, the hypothesis of independence between $f_X$ and $g_Y$ is expressed by the following probability density function over $I\times_M J$: 
\[
( f\star g) (x,y)=
\begin{cases}
\displaystyle \frac{f(0) \ g(0)}{f(0)} & \text{ if } x=y = 0 \\[15pt]
\displaystyle \frac{f(x)\ g(y)}{1-f(0)} & \text{ otherwise.} 
\end{cases}
\]

\end{example}

\begin{example}[Independence of discrete random variables]

If in the previous example we do not exclude the elements $(x, 0)$ and $(0, y)$, then the standard Cartesian product $I \times J$ is the state space and the independence assumption corresponds to the usual factorization $f_X \star g_X = f_Xg_X$, which is a special case of a Markov combination where the set of meta-categories consists of a single element.

\end{example}

\begin{example}
For a more more abstract example let $J = \{-3,-2,-1,0,1\}$
and $I=\{0,1,2,3\}$. Consider 
a two-element set of metacategories $M = \{a,b\}$ and two category mappings $I\to M$, $J\to M$ where $0,1\mapsto a$ and the other elements of $I$ and $J$ map to $b$.
The mapping product is
\[
I \times_M J = (I_a\times J_a) \cup (I_b\times J_b)
\]
where $I_a=J_a=\{0,1\}$, $I_b = \{2,3\}$, and $J_b=\{-3,-2,-1\}$. Two distributions $f$ and $g$ over $I$ and $J$, respectively, are consistent if $f(0)+f(1) = g(0) + g(1)$. Their Markov combination is
\[
(f\star g)(i,j) = \frac{f(i) \ g(j)}{\sum_{i'\in I_k}f(i')}
\quad \text{for all } k\in M, i\in I_k, j\in J_k.
\]
\end{example}

\begin{example}[Distributions with given marginals]\label{generalisation} 
Consider three finite sets $A,B,C$, 
and define $I=A\times C$ and $J=C\times B$. 
Assume two distributions $f:I\to \mathbb R$ and $g:J\to \mathbb R$ such that the marginals over $C$ coincide: $f_C = g_C$. (Recall from Example~\ref{marginal} that marginals are aggregates). We would like to define a distribution $h$ on $A\times C\times B$ such that $h_I = f$ and $h_J=g$. It turns out that $h$ can be realised as a Markov combination of $f$ and $g$, once appropriate meta-categories are defined.

Let $M=C$ and define
the category mappings $p:A\times C\to C$ and $q:C\times B\to C$
 as the projections to $C$. 
Then $I\times_M J$ consists of all pairs $((a,c),(c',b))$ where $c=c'$. Therefore, $I\times_M J\simeq A\times C\times B$. Under this identification, the Markov combination of $f$ and $g$ is
\[
(f\star g)(a,c,b) = \frac{f(a,c)g(c,b)}{f_C(c)} = \frac{f(a,c)g(c,b)}{\sum_{a'\in A}f(a',c)}.
\]
The marginal distributions of $f\star g$ satisfy the required conditions 
\begin{align*}
\sum_{b'\in B}(f\star g)(a,c,b')
&= \frac{f(a,c)g_C(c)}{f_C(c)}
= f(a,c)
\quad \text{and}\\
\sum_{a'\in A}(f\star g)(a',c,b)
&= \frac{f_C(c)g(c,b)}{f_C(c)}
= g(c,b)
\end{align*}
for all $a\in A$, $b\in B$, $c\in C$.
\end{example}

\begin{remark}
Example~\ref{generalisation} shows that our discrete Markov combinations are generalisations of the Markov combinations of~\cite{massa-riccomagno}, applied to the discrete case.
\end{remark}

\begin{remark}\label{aggregate-meta-markov}
If $p:I\to M$ and $q:J\to M$ are category mappings and $f\in K(\Theta)^I$, $g\in K(\Theta)^J$ are meta-consistent statistical models, then
\[
(f\star g)_I = f \quad \text{and} \quad (f\star g)_J = g,
\]
where the aggregates are taken with respect to the natural category mappings $\pi_I:I\times_M J\to I$ and $\pi_J:I\times_M J\to J$. For instance, for $i\in I$ we have
\[
(f\star g)_{I,i} = 
\sum_{\pi_I(i',j') = i} \frac{f_{i'}\,g_{j'}}{f_{M,p(i')}}
=
\sum_{q(j) = p(i)} \frac{f_i g_j}{f_{M, p(i)}}
=
\frac{f_i}{f_{M,p(i)}}\, g_{M,p(i)} = f_i.
\]
This generalises the fact that for meta-consistent statistical models defined on overlapping Cartesian spaces as in Example~\ref{generalisation}, the operation of Markov combination preserves the margins. 
\end{remark}

\section{Properties}\label{sec:properties}

\subsection{Associativity}

Let $f$, $g$, and $h$ be distributions on the finite sets $I$, $J$, and $K$, respectively.
The possibilities for combining $f,g,h$ via Markov combinations depend on the meta-category structure on $I$, $J$ and $K$.
Here we shall study these possibilities in the case of meta-Markov combinations.

Assume that $f$, $g$ and $h$ are pairwise meta-consistent~\cite{massa-lauritzen}. This means that there are meta-categories $M_1$, $M_2$, and $M_3$ and three pairs of category mappings relating each pair of category sets such that the corresponding meta-consistence relations hold, as shown in the following diagram:

\[\begin{tikzcd}
	& I \\
	{M_1} && {M_3} \\
	J && K \\
	& {M_2}
	\arrow[from=1-2, to=2-1]
	\arrow[from=1-2, to=2-3]
	\arrow[from=3-1, to=2-1]
	\arrow[from=3-1, to=4-2]
	\arrow[from=3-3, to=2-3]
	\arrow[from=3-3, to=4-2]
\end{tikzcd}
\quad\quad
\begin{tabular}{ScScSc}
$f_{M_1} = g_{M_1}$\\
$g_{M_2} = h_{M_2}$\\
$h_{M_3} = f_{M_3}$
\end{tabular}
\]

Taking into account that the operator $\star$ is commutative, there are six possible ways of obtaining a distribution from $f$, $g$, and $h$ by applying a binary meta-Markov combination twice.
\begin{alignat*}{2}
(f\star_{M_1} g) \star_{M_3} h
&\qquad (f \star_{M_1} g) \star_{M_2} h
&\qquad f \star_{M_3} (g \star_{M_2} h)
\\
(f\star_{M_3} h) \star_{M_1} g
&\qquad f\star_{M_1} (g \star_{M_2} h)
&\qquad f\star_{M_3} (h \star_{M_2} g)
\end{alignat*}
These are all well defined. Indeed,
we can write the first combination due to the category mapping $I\times_{M_1} J \to M_3$ given by passing through $I$, and the second combination can be taken via the category mapping $I\times_{M_1} J\to M_2$ passing through $J$. Similar arguments apply to the others.

By expanding these  expressions and using the fact that for a chain of category mappings $I\to M\to M'$ and a distribution $f$ on $I$ we have $(f_M)_{M'}= f_{M'}$,
we obtain the following equalities:
\begin{align*}
(f\star_{M_1} g) \star_{M_3} h
&= (f\star_{M_3} h) \star_{M_1} g
\\
(f \star_{M_1} g) \star_{M_2} h
&= f\star_{M_1} (g \star_{M_2} h)
\\
f \star_{M_3} (g \star_{M_2} h)
&= f \star_{M_3} (h \star_{M_2} g).
\end{align*}
In general, there are no more equalities among these six expressions. Indeed, the equality $(f\star_{M_1} g) \star_{M_3} h = (f \star_{M_1} g) \star_{M_2} h$
is equivalent to the existence of a bijection $M_3\simeq M_2$ that makes the above diagram commute, and such that $h_{M_3} = h_{M_2}$ under this bijection. Similar statements hold for the other two missing equalities.

If $M\coloneqq M_1=M_2=M_3$, then the meta-consistency hypothesis is $f_M=g_M=h_M$. In this case, all six combinations are equal and full associativity holds:
\[
(f \star_M g) \star_M h = f \star_M (g \star_M h) = (f \star_M h) \star_M g.
\]
More generally, if we are given $m$ category mappings $I_\ell \to M$ and pairwise meta-consistent distributions $f_\ell$ on $I_\ell$ (for $\ell\in\{1,\ldots,m\}$), the $m$-ary meta-Markov combination $\star_{\ell=1}^m f_\ell$ is defined by taking successive binary meta-Markov combinations of the $f_\ell$. By associativity, every possible way to do this gives the same distribution.

\subsection{Unit Element, Inverses} The unit element with respect to $\star$ is the degenerate model $\mathbf{1}$ corresponding to the unique probability density function over the set with one element.
In general, there are no inverses with respect to $\star$ since the set of categories of a meta-Markov combination always has at least as many elements as the sets of categories of its components.

\subsection{Mixtures} To demonstrate the flexibility of general Markov combinations combined with aggregates, we show how to construct arbitrary mixture models using only these two operations. Recall that for two parametric models $(f_i)$, $(g_i)$ $(i\in I)$, the \emph{mixture} $\Mixt(f,g)$ is defined as the model $(\lambda f_i + (1-\lambda)g_i)$ $(i\in I)$, where $\lambda\in[0,1]$ is an additional parameter. 

\begin{proposition}
Let $f$ and $g$ be two parametric models over the same finite set $I$.
Let $\mathbf 2$ denote the one-parameter model defined by sending $\lambda\in [0,1]$ to $(\lambda, 1-\lambda)$. 
Then $\Mixt(f,g)$ can be constructed from $f$, $g$, and
$\mathbf 2$
by a chain of Markov combinations and aggregates as follows: 
\[
((f\times \mathbf 2)_{M_1} \star (g\times \mathbf 2)_{M_2})_{M_4} = (\lambda f_i + (1-\lambda)g_i\mid i\in I) = \Mixt(f,g).
\]
\begin{proof}
Construct
the models \[f\times \mathbf 2 = (\lambda f_i, (1-\lambda) f_i \mid i\in I) \quad\text{and}\quad g\times \mathbf{2} = (\lambda g_i, (1-\lambda) g_i \mid i\in I)\] as Markov combinations with the trivial category mapping $I\to \{\bullet\}$. Aggregate the components of $f\times \mathbf{2}$ involving $1-\lambda$ and those of $g\times \mathbf{2}$ involving $\lambda$ to obtain
\[
(f\times \mathbf 2)_{M_1} = ((\lambda f_i\mid i\in I), 1-\lambda) \quad\text{and}\quad (g\times \mathbf 2)_{M_2} = \bigl(\lambda, ((1-\lambda)g_i \mid i\in I)\bigr).
\]
Next, define category mappings to $M_3=\{0,1\}$ that put all $\lambda f_i$ in the same category as $\lambda$ and all $(1-\lambda)g_i$ in the same category as $(1-\lambda)$. With these mappings, the distributions $(f\times \mathbf 2)_{M_1}$ and $(g\times \mathbf 2)_{M_2}$ are meta-consistent since $\sum_i \lambda f_i = \lambda$ and $\sum_i (1-\lambda)g_i = 1-\lambda$. Their Markov combination is
\[
(f\times \mathbf 2)_{M_1} \star (g\times \mathbf 2)_{M_2}
=
(\lambda f_i, (1-\lambda) g_i \mid i\in I).
\]
Aggregating the entries $\lambda f_i$ and $(1-\lambda)g_i$ pairwise yields
\[
((f\times \mathbf 2)_{M_1} \star (g\times \mathbf 2)_{M_2})_{M_4} = (\lambda f_i + (1-\lambda)g_i\mid i\in I) = \Mixt(f,g). \qedhere
\]
\end{proof}

\begin{remark}
Mixture models have an $m$-ary variant defined as
\[
\Mixt(f^{(j)} \mid j\in [m])_i = 
\sum_j \lambda_j f^{(j)}_i,
\]
where $f^{(1)},\dotsc,f^{(m)}$ are models over $I$ and the $\lambda_j$ are non-negative parameters such that $\sum_j \lambda_j = 1$. Up to a reparametrisation of the mixture parameters, these models are equivalent to successive binary mixtures, for instance $\Mixt(\Mixt(f,g),h)\simeq \Mixt(f,g,h)$. Therefore, $m$-ary mixtures can be obtained by concatenating binary Markov combinations and aggregates.
\end{remark}

\end{proposition}

\section{Special Models}\label{sec:models}

\subsection{Discrete Copulas}
In this subsection, we establish a relationship between discrete copulas in the sense of~\cite{discrete-copulas} and Markov combinations. We will show that \emph{the product of two discrete copulas corresponds to a marginal of the Markov combination of their distributions}.

Let $n,m\in\mathbb N$ and define $\langle n\rangle \coloneqq\{0,,\dotsc,n\}$, $[n]\coloneqq\{1,\dotsc,n\}.$ Following the presentation of Kolesárová et al. with $\langle n\rangle$ in place of $I_n$ therein, a \emph{discrete copula} is a function
\[
C: \langle n\rangle \times \langle m\rangle\to [0,1]
\]
such that for all $i\in\langle n\rangle$ and $j\in\langle m\rangle$,
\begin{enumerate}
\item[(C1)] $C(i,0) = C(0,j) = 0$,
\item[(C2)] $C(i,m) = i/n$, $C(n,j) = j/m$,
\item[(C3)] $C(i-1,j-1)-C(i,j-1)-C(i-1,j)+C(i,j)\geq 0$ whenever $i,j\geq 1.$
\end{enumerate}

Recall that a matrix $(a_{ij})\in\mathbb R^{n\times n}_{\geq 0}$ is \emph{bistochastic} if $\sum_{i=1}^n a_{ik} = \sum_{j=1}^n a_{kj} = 1$ for all $k\in[n]$.
For a function $C:\langle n\rangle\times\langle n\rangle\to[0,1]$, the following are equivalent:
\begin{enumerate}
\item[(a)] $C$ is a copula;
\item[(b)] there is a bistochastic matrix $A = (a_{ij})\in \mathbb N_{\geq 0}^{n\times n}$ such that
\[
C(i,j) = \frac{1}{n}\sum_{k=1}^{i}\sum_{\ell=1}^j a_{k\ell};
\]
\item[(c)] there is a probability density $h:[n]\times [n]\to[0,1]$ such that for all $i\in[n]$ and $j\in[n]$ 
\[
\sum_{k=1}^n h(k,j) = \sum_{\ell=1}^n h(i,\ell) = 1/n \quad \text{and} \quad
C(i,j) = \sum_{k=1}^i\sum_{\ell=1}^j h(k,\ell).
\]
\end{enumerate}
The equivalence between (a) and (b) is proven in~\cite{discrete-copulas}. The equivalence with (c) follows by recalling that copulas are cumulative density functions of multivariate random variables with uniformly distributed marginals and by setting 
$h(k,\ell) = a_{k,\ell}/n = $ the left-hand term of the inequality in (C3).

To establish a relationship between products of discrete copulas and Markov combinations,  let $C_1,C_2:\langle n\rangle^2\to [0,1]$ be copulas, $A=(a_{ij}), B=(b_{ij})$ their associated bistochastic matrices and $\alpha, \beta$ their associated distributions on $[n]^2$ given by $a_{ij} = n\alpha(i,j)$ and $b_{ij}=n\beta(i,j)$. The authors of~\cite{discrete-copulas} define the \emph{product copula} 
as the copula corresponding to the bistochastic matrix $A\cdot B$. A computation shows that the distribution $\gamma$ associated to the product copula is given by
\[
\gamma(i,j) = \sum_{k=1}^n n \alpha(i,k)\beta(k,j) \quad\text{for all } i,j\in [n].
\]
Now let $I=J=[n]\times [n]$, $M=[n]$, and define the category mappings $p:I\to M$ and $q:J\to M$ as the second and first projection to $[n]$, respectively. Consider $\alpha$ and $\beta$ as distributions on $I$ and $J$, respectively. For $k\in M$ we have $\alpha_M(k) = \beta_M(k) = 1/n$ and therefore
\[
\gamma(i,j) = \sum_{k=1}^n\frac{\alpha(i,k)\beta(k,j)}{\alpha_M(k)} = \sum_{k=1}^n (\alpha\star\beta)(i,k,j).
\]
Thus, $\gamma$ is a marginal distribution, i.e.\ an aggregate, of the Markov combination of $\alpha$ and $\beta$. If we define $M'=[n]\times [n]$
and the
category mapping  $I\times_M J = [n]\times [n]\times [n] \to M'$ by sending $(i,k,j)$ to $(i,j)$, we can write this relationship more succintly as
\[
\gamma = (\alpha\star \beta)_{M'}.
\]

\begin{remark}

The above discussion suggests a slight generalisation of copulas, where in (c), the marginals of the density $h$ are not $1/n$ but any two fixed distributions $\nu,\mu$ on $[n]$ depending on the direction of the marginal, i.e.
\[
\sum_{k=1}^n h(k,j) = \mu(j) \quad\text{and}\quad
\sum_{\ell=1}^n h(i,\ell) = \nu(i).
\]
Then, condition (C2) would be replaced by
\[
C(i,m) = \sum_{k=1}^i\nu(k) \quad\text{and}\quad
C(n,j) = \sum_{\ell=1}^j\mu(\ell).
\]
Indeed, given (C1--C3) with the modified (C2), one can verify by telescope sums that the left-hand side $h(i,j)$ of (C3) defines a probability distribution on $[n]\times [n]$ with marginals $\mu$ and $\nu$ such that $C$ is the cumulative distribution of $h$.
\end{remark}

\subsection{Saturated models}\label{sec:saturated-models}
In the discrete setting,
a \emph{saturated model is}
a parametric model $f=(f_i)_{i\in I}$ where the image of $f\colon \Theta\to \mathbb R^I$ equals the probability simplex $\Delta^{|I|-1}$.
This terminology is used for instance in~\cite{sullivant}.
A standard parametrisation of any saturated model on $I=\{0,1,\ldots,n\}$ is given as follows: $f = (\theta_0,\theta_1,\dotsc,\theta_n)$, $\theta_1,\dotsc,\theta_n \geq 0$, $\theta_0 \coloneqq 1 - \sum_{i>0}\theta_i$. However the following example shows that other parametrisations might be better for Markov combination.

\begin{example}
Let $f = (\theta_0,\theta_1,\theta_2,\theta_3)$ and $g=(\eta_0,\eta_1,\eta_2,\eta_3)$ be two saturated models. Let $I=J=\{0,1,2,3\}$ and define category mappings $I\to \{0,1\}$ and $J\to \{0,1\}$ by setting $I_0=\{0,1\}$, $I_1=\{2,3\}$, $J_0=\{0\}$, $J_1=\{1,2,3\}$. Then $f$ and $g$ are not meta-consistent, as this would require
\begin{align*}
\theta_0+\theta_1 &= \eta_0\\
\theta_2+\theta_3 &= \eta_1+\eta_2+\eta_3.
\end{align*}
However, the above equations are satisfied by reparametrising $g$: set \[g' = (\theta_0+\theta_1, \eta_1, \eta_2, \theta_2+\theta_3-\eta_1-\eta_2),\] where $\eta_1, \eta_2$ are two new parameters $\geq 0$. Then $f$ and $g'$ are meta-consistent.

For a more systematic approach, introduce new parameters $\lambda_1,\lambda_2,\lambda_3 \geq 0$ such that $\lambda_1+\lambda_2+\lambda_3 = 1$ and set
\[
g' = (\theta_0+\theta_1, \lambda_1(\theta_2+\theta_3),\lambda_2(\theta_2+\theta_3),\lambda_3(\theta_2+\theta_3)).
\]
Then $g'$ is saturated and meta-consistent with $f$.
\end{example}
This example suggests the more general result of Proposition~\ref{make-meta-consistent}.

\begin{proposition}\label{make-meta-consistent}
For all pairs of category mappings $I\to M$ and $J\to M$, any two saturated models can be made meta-consistent after reparametrisation.
\end{proposition}

We prove Proposition~\ref{make-meta-consistent} by chaining together the following constructions.

\begin{construction}\label{model-lift}
\emph{Given a category mapping $I\to M$ and a model $h$ over $M$, we construct a model $f$ over $I$ such that $f_M = h$.}

For $k\in M$, let $\lambda_{k,i}$ be new parameters such that $\sum_{i\in I_k}\lambda_{k,i} = 1$ and for $i\in I_k$, let $f_i=\lambda_{k,i} h_k$. Then $\sum_{i\in I_k} f_i = h_k$ and $\sum_{i\in} f_i = 1$. Therefore, $f$ is a model over $I$ such that $f_M = h$.\qed
\end{construction}

\begin{lemma}\label{lift-is-saturated}
In Construction~\ref{model-lift}, if $h$ is saturated, then $f$ is saturated. 
\end{lemma}
\begin{proof}
Let $(a_i)\in\Delta^{|I|-1}$. Since $h$ is saturated, there is a choice of parameters such that $a_{M,k}=h_k$ for all $k\in M$. Now choose $\lambda_{k,i}=a_i/a_{M,k}$. With these parameters, $f_i=a_i$ for all $i\in I$.
\end{proof}

\begin{construction}\label{saturated-consistent}
\emph{Given category mappings $I\to M$ and $J\to M$, we construct saturated models over $I$ and $J$ that are meta-consistent with respect to these mappings.}

Let $h$ be the saturated model over $M$ given by the standard parametrisation $h=(\theta_k)_{k\in M}$. Use Construction~\ref{model-lift} to define models $f$ over $I$ and $g$ over $J$ such that $f_M = h = g_M$. By definition, $f$ and $g$ are meta-consistent. Both are saturated by Lemma~\ref{lift-is-saturated}.
\end{construction}

\begin{proof}[Proof of Proposition~\ref{make-meta-consistent}]
Let $f'$ and $g'$ be the given saturated models on $I$ and $J$, respectively.
Use Construction~\ref{saturated-consistent} to define saturated models $f$ on $I$ and $g$ on $J$ that are meta-consistent. These are the required reparametrisations of $f'$ and $g'$. Should it be preferred to leave $f'$ as-is and to only reparametrize $g'$, apply Construction~\ref{model-lift} to $J\to M$ with $h=g'_M$ to get the desired reparametrisation $g$.  
\end{proof}

\begin{example}\label{saturated-consistent-example}
Construction~\ref{saturated-consistent} gives meta-consistent $f$ and $g$ with respect to any given pair of category mappings $I\to M$ and $J\to M$. What is their meta-Markov combination? By construction, for all $k\in M$, $i\in I_k$ and $j\in J_k$ we have
\begin{equation}\label{saturated-markov}
(f\star g)_{i,j} = \lambda_{k,i}\mu_{k,j}h_k
\end{equation}
where $\lambda_{k,i}$, $\mu_{k,j}$, $h_k$ are non-negative with $\sum_{i\in I_k}\lambda_{k,i} = \sum_{j\in J_k}\mu_{k,j} = \sum_k h_k = 1$. It follows that $f\star g$ is a special kind of \emph{mixture} of independence models, which we call a ``pure mixture''. More precisely, we have for each $k$ an independence model $\Delta^{|I_k|-1}\times \Delta^{|J_k|-1}$ embedded into $\Delta^{|I_k||J_k|-1}$ via the Segre embedding. Each of these $\Delta^{|I_k||J_k|-1}$ is taken to be the facet of a larger simplex, pairwise disjoint from all other such facets. Then the $|M|$-mixture of these independence models is taken.
\end{example}

\begin{example}
For a more concrete example, let $I\to M$ and $J\to M$ be as the start of the section, i.e.\ with 
$I_0=\{0,1\}$, $I_1=\{2,3\}$, $J_0=\{0\}$, and $J_1=\{1,2,3\}$.
Then
\begin{align*}
f &= (\lambda_{00}h_0, \lambda_{01}h_0, \lambda_{12}h_1,\lambda_{13}h_1),\\
g &= (\mu_{00}h_0, \mu_{11}h_1, \mu_{12}h_1,\mu_{13}h_1),\\
f\star g &= (\lambda_{00}\mu_{00}h_0, \lambda_{01}\mu_{00}h_0,\\
&\phantom{= ( } \lambda_{12}\mu_{11}h_1, \lambda_{12}\mu_{12}h_1, \lambda_{12}\mu_{13}h_1,
  \lambda_{13}\mu_{11}h_1, \lambda_{13}\mu_{12}h_1, \lambda_{13}\mu_{13}h_1).
\end{align*}
Thus, $f\star g$ is the ``pure mixture'' of $\Delta_1$ and $\Delta_1\times \Delta_2$ as described above.
\end{example}

\begin{example}
Let us analyze variant (3a) of Definition~\ref{markov-variants}  with respect to the standard saturated models $f = (\theta_i)$ and $g = (\eta_j)$. Since $f$ and $g$ share no parameters, there is no difference between (3a) and ($3'$a) of Definition~\ref{restricted-markov-variants}. The parameters $(\theta,\eta)$ of $f\lowerstar g$ map surjectively onto the parameters $(\lambda_i,\mu_j,h_k)$ of Example~\ref{saturated-consistent-example} by taking $h_k = \sum_{i\in I_k} \theta_k$, $\lambda_i = \theta_i/h_k$, $\mu_i = \eta_j/h_k$. Furthermore, for each $k$, the entries of $f\lowerstar g$ over $I_k\times J_k$ satisfy the equations of the independence model $\Delta^{|I_k|-1}\times \Delta^{|J_k|-1}$. Therefore, the image of $f\lowerstar g$ is contained in the image of $f\star g$. Because of the surjective mapping of parametrs, $\im(f\star g) = \im(f\lowerstar g)$. So, in this case, $(2) = (3\text{a}) = (3'\text{a})$.
\end{example}

\subsection{Regular exponential families}
Discrete, regular exponential families $f = (f_i)_{i\in I}$ differ from saturated models in that a) the interior of $\Delta^{|I|-1}$ equals the image of $f:\Theta\to \Delta^{|I|-1}$ and b) the dimension of the parameter space $\Theta$ equals $|I|-1$, so that $\Theta$ can be transformed into the space of canonical parameters of the family.
Since the constructions in Subsection~\ref{sec:saturated-models} preserve these properties, the discussion of Markov combinations of regular exponential families follows the lines of that subsection. In particular, the Markov combination of two exponential families over $I$ and $J$ with metacategories $M$ is the interior of a pure mixture of independence models. The dimension of its parameter space is computed as follows. Following Equation~\eqref{saturated-markov}, for each $k\in M$ there are $1 + |I_k| + |J_k|$ parameters and two linear constraints. Additionally, there is one linear constraint for the $h_k$. Summing over $k$ yields $|M|+|I|+|J| - 2|M| - 1 = |I|+|J|-|M|-1$ free parameters. The dimension of the ambient space is $\sum_{k\in M}|I_k|\cdot|J_k|-1$ as explained after Definition~\ref{mapping-product}.

\subsection{Staged Tree Models}
A \emph{staged tree model} is a probabilistic graphical model for finite sample space processes introduced in~\cite{smith-anderson} to model non-symmetric independence and causality~\cite{thwaites} among other applications.  It is supported on a \emph{staged tree} and is a curved exponential family~\cite{goergen-leonelli}. For further details on the definition and usage of staged trees, see~\cite{goergen-bigatti}, \cite{collazo-smith}, \cite{leonelli-varando}, \cite{carli-leonelli}, and references therein.
Here, we recall that a staged tree is 
 defined by a directed rooted tree   $T=(V,E) $ with a finite set of vertices $V$ and whose edges $E$ are  labeled by the elements of a label set $L$. 
For each vertex $v\in V$, the set of edges $ E(v)=\{e= (v,w)\in E \ | \ w\in V\} $ is called a \emph{floret}  and the labeling $\{\theta_e \in L \mid e\in E\}$ must satisfy the following condition:
\begin{itemize}
\item for any two florets $E(v)$ and $E(v')$, the sets \[\{\theta_e\mid e\in E(v)\} \quad\text{and}\quad \{\theta_e\mid e\in E(v')\}\] are either equal or disjoint.
\end{itemize}

Figure~\ref{tree1} depicts a staged tree with edge labels $\theta_0, \theta_1, \theta_2, \theta_3$ which, when interpreted as transition probabilities, satisfy the linear condition $\theta_0+\theta_1=\theta_2+\theta_3=1$ and are non-negative.

\begin{figure}
\includegraphics[width=0.7\textwidth,trim={0 10em 0 7em},clip]{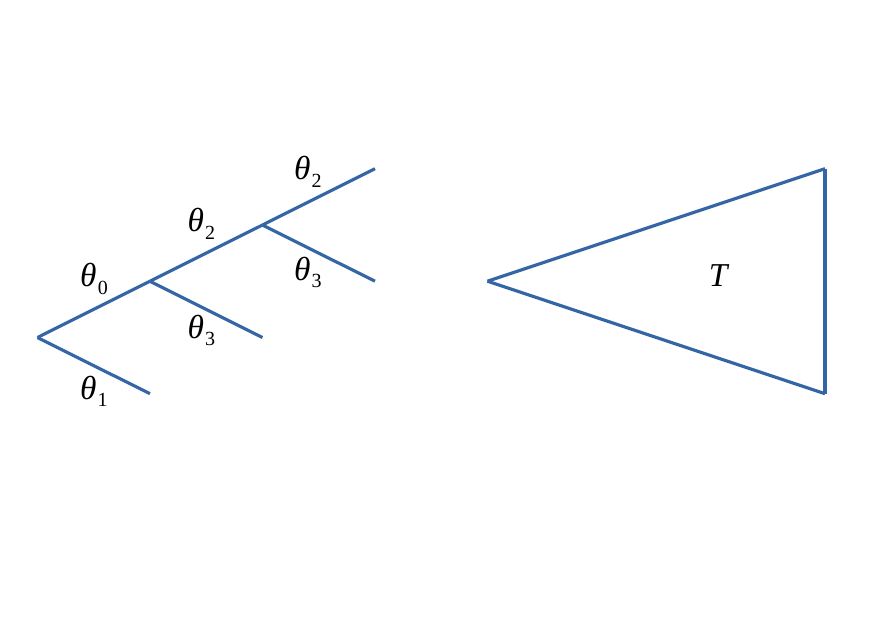}
\caption{A typical staged tree $T$. On the left, the structure of the tree and its edge labels. The tree has three florets, two of which have the same edge labels. On the right, a schematic representation for the whole tree.} \label{tree1}
\end{figure}
  
A discrete parametric statistical model based on a stage tree $T$, called \emph{a staged tree model}, can be defined on the set $I$ of all root-to-leaf paths of $T$ by taking the labels $\theta_e\in [0,1]$ as parameters and 
setting 
\[
\sum_{e\in E(v) } \theta_e = 1 \text{ for all } v\in V
\quad \text{and} \quad
f_i = \prod_{e\in E(i)} \theta_e \text{ for all } i\in I 
\]
where $E(i)$ is the set of edges in the root-to-leaf path $i\in I$.
Whereas the edge labels $\theta_e$ play the role of transition probabilities and are the model parameters, the $f_i$ play the role of joint probabilities and define the statistical model itself.

Every staged tree $T$ can be decomposed as a collection of subtrees $T_1,\dotsc,T_m$ attached to the leaves of a root tree $S$. We denote such a decomposition by
\begin{equation}\label{staged-decomposition}
(S, T_1, \dotsc, T_m).
\end{equation}
For example, the root tree $S$ can be the root vertex of $T$, or a subtree $T_k$ can be a leaf of $T$.
The subtrees $T_k$ are staged trees with disjoint vertex sets $V(T_k)\subset V(T) $ and edge sets $E(T_k)\subset E(T) $, $k=1,\ldots,m$. For the vertex and edge sets of $S$ we write $V(S)$ and $E(S)$, respectively. There are many decompositions for every given staged tree, and every such decomposition induces a mapping $\{\text{Leaves of } T\}\to\{\text{Leaves of } S\}$ by assigning to each leaf of $T$ its closest ancestor among the leaves of $S$.
This defines an equivalent category mapping $p:\{\text{Root-to-leaf paths of } T\}\to \{\text{Root-to-leaf paths of } S\}$
that assigns to $i\in I$ the unique root-to-leaf path $p(i)$ of $S$ contained in $i$.

Any decomposition~\eqref{staged-decomposition} induces a factorisation of the model coordinates: for every root-to-leaf path $i$ of $T$,
\begin{equation}\label{coord-decompose}
f_i = \left(\prod_{e\in E(i)\cap E(S)} \theta_e\right) \left(\prod_{e\in E(i)\cap E(T_{p(i)})} \theta_e\right) = s_{p(i)}\, t_i,
\end{equation}
where $T_{p(i)}$ is the unique subtree in the decomposition that contains a subpath of $i\in I$ and $s_{p(i)}$, $t_i$ are the factors in the first and second parentheses, respectively. Note that $s_{p(i)}$ depends only on 
$p(i)\in\{\text{Root-to-leaf paths of } S\}$.

Let $T$ and $T'$ be staged trees and $(S, T_1,\dotsc,T_m)$ and $(S',T_1',\dotsc,T_{m}')$ decompositions of $T$ and $T'$ with the same number of subtrees. A pair of category mappings \[\{\text{Leaves of } T\}\to\{\text{Leaves of } S\}\leftarrow \{\text{Leaves of } T'\}\]
can be realised via a bijection $\varphi : \{\text{Leaves of } S\}\to \{\text{Leaves of } S'\}$ which we denote by $k\mapsto k'$.
Then for $M = \{\text{Leaves of } S\}$
we have the following characterisation of meta-consistence, see Figure~\ref{tree2} for a schematic representation.

\begin{figure}
\includegraphics[width=0.7\textwidth,trim={0 7em 0 7em},clip]{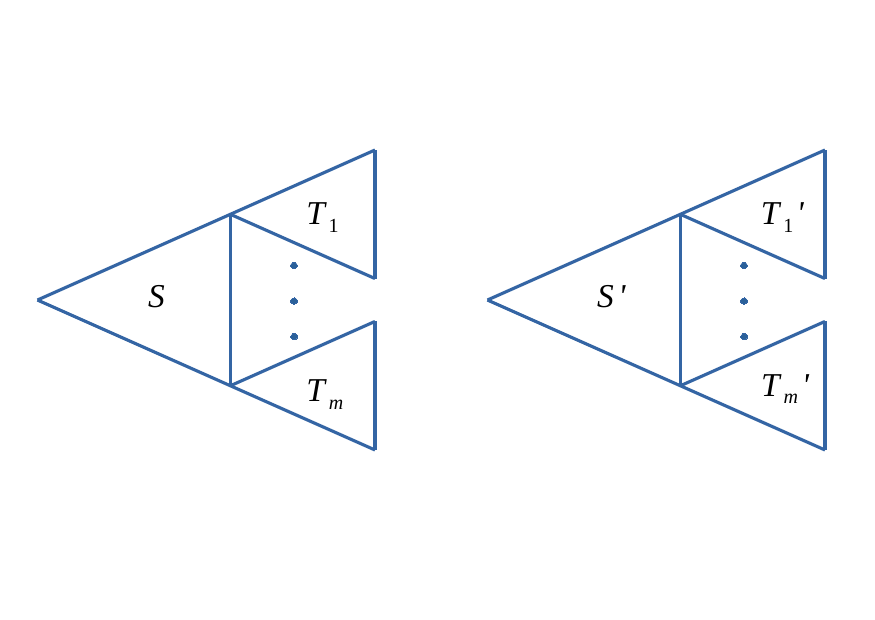}
\caption{Meta-consistent decompositions of two staged trees, represented as $m$ subtrees $T_i$ resp.\ $T_i'$ attached to the $m$ leaves of a rooted subtree $S$ resp.\ $S'$. The trees $S$ and $S'$ are not necessarily equal, but have the same number of leaves and represent the same statistical model on $m$ categories.} \label{tree2}
\end{figure}

\begin{proposition}
The staged tree models of $T$ and $T'$ are meta-consistent if and only if the probabilities of the root-to-leaf paths of the subtrees $S$ and $S'$ are equal in the model, that is
\[
\prod_{e\in E(k)}\theta_e = \prod_{e'\in E(k')}\theta'_{e'} \quad
\text{for all root-to-leaf paths } k \text{ of } S.
\]
\end{proposition}
With respect to the factorisations $f_i = s_{p(i)}\,t_i$ and $f'_j=s'_{q(j)}\,t'_j$ as in~\eqref{coord-decompose}, this condition is equivalent to 
$s_{p(i)} =s'_{q(j)}$
whenever $q(j) = \varphi(p(i))\in M$.

\begin{proof}
Let $f$ be the model of $T$
and $f'$ the model of $T'$.
For  root-to-leaf paths $i$
of $T$ and $j$ of $T'$,
write $f_i = s_{p(i)}\, t_i$
and $f'_j = s'_{q(j)}\, t'_j$
as in~\eqref{coord-decompose}.
For all
root-to-leaf paths
$k$ of $S$ we have $\sum_{p(i)=k} t_i=1$. This can be shown either inductively, starting with the case where $T_i$ has height one (a floret) or by noting that $t_i$ is the conditional probability of the path $i$ given that $i$
contains the path $p(i)$.
Since  $s_{p(i)}$ takes the same value, say $s_k$, for all $i$ such that $p(i)=k$, we have
\[
f_{M,k} = \sum_{p(i)=k} s_{p(i)}\, t_i = s_k \sum_{p(i)=k} t_i = s_k.
\]
Therefore, $f$ and $f'$ are meta-consistent if and only if $s_k = s'_{\varphi(k)}$ for all root-to-leaf paths $k$ of $S$.
We conclude the proof by noting that $s_k=\prod_{e\in E(k)} \theta_e$, and similarly for $s_{\varphi(k)}$.
\end{proof}

\begin{remark}
The subtrees $S$ and $S'$ need not be equal for the models of $T$ and $T'$ to be meta-consistent, but they should have the same number of leaves and consistent probabilities. For instance, suppose that $S$ is a binary tree with four leaves representing the joint probability space of two binary random variables $X_1$ and $X_2$. Switching the order of the variables and changing the edge labels accordingly yields a tree $S'$ different from $S$ but with the same root-to-leaf paths.
\end{remark}

Let $T$ and $T'$ be meta-consistent, 
 $(S,T_1,\ldots,T_m)$ and $(S',T_1',\ldots,T_m')$ their decompositions, and $\varphi$ the bijection from the leaves of $S$ to the leaves of $S'$. Then their Markov combination $T''$ is obtained by taking $T$ and attaching a copy of $T'_{\varphi(j)}$ to each leaf of $T_j$, for $j=1,\ldots,m$. The staged tree $T''$ can therefore be decomposed as $(S,T_1'',\dotsc,T_m'')$ where $T_j''=(T_j, T_{\varphi(j)}',\dotsc, T_{\varphi(j)}')$.

To see this, for each root-to-leaf path $i$ in $T$ write its probability as $f_i = s_{p(i)}\, t_i$ and similarly write $f'_j = s'_{q(j)} t'_j$ for each root-to-leaf path $j$ of $T'$ in the same metacategory of $i$. Then, 
$s_{p(i)} = s'_{q(j)}$ because of meta-consistency and thus $(f\star f')_{i,j} = 
f_i\, f'_j / s_{p(i)} = t_i\, s_{p(i)}\,t'_{j}$ where the first equality follows by definition of $\star$.

Note that, depending on whether we prefer the formula for a right- or left Markov combination, the roles of $T$ and $T'$ can be reversed: an equally valid staged tree representation for their Markov combination would be
\[
(S',(T'_1, T_1,\dotsc,T_1),\dotsc,(T'_m,T_m,\dotsc,T_m)).
\]
This results in combinatorially distinct trees that are nevertheless statistically equivalent (they give rise to the same staged tree model, see~\cite{goergen-bigatti}). A schematic depiction of these Markov combinations is found in Figure~\ref{tree3}.

\begin{figure}
\includegraphics[width=0.7\textwidth,trim={0 4em 0 4em},clip]{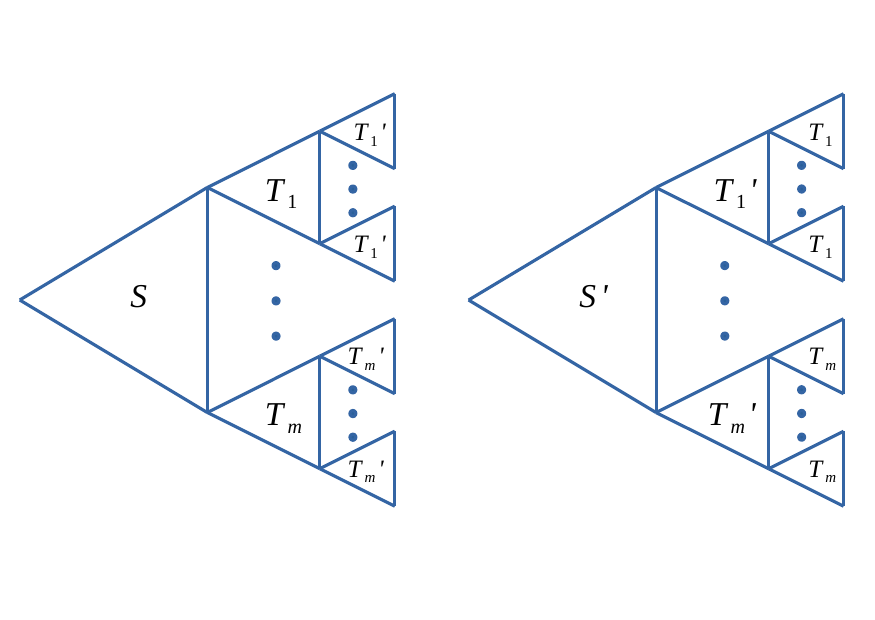}
\caption{Two statistically equivalent Markov combinations of the two staged trees in Figure~\ref{tree2}. On the left, the ``tail'' subtrees of the second tree are attached to the leaves of the first tree. On the right, the roles of the first and second tree are reversed.}\label{tree3}
\end{figure}

\section{Theorems}\label{sec:theorems}

\subsection{Model invariance}

To formulate the principle of model invariance as explained e.g.\ in~\cite[\S 6.4]{casella-berger}, one
starts by considering a group $G$ acting on the set of elementary events. In the finite discrete case, each element of the group $G$ gives a permutation of the set of categories $I$.
We write this action as $G\times I\to I$, $(\alpha,i)\mapsto \alpha\cdot i$. This induces an action on $\Delta^I$ by
\[
(\alpha \cdot v)_i \coloneqq v_{\alpha^{-1}(i)}
\]
for all $\alpha\in G$, $v\in\Delta^I$, $i\in I$.

\begin{definition}
Let $G\times I\to I$ be a group action and $\alpha\in G$. A statistical model $f:\Theta\to \Delta^I$ is \emph{$\alpha$-invariant} if there exists an injective function $\overline\alpha:\Theta\to \Theta$ such that $\alpha\circ f = f\circ\overline\alpha$. That is: $f(\theta)_{\alpha^{-1}(i)} = f(\overline\alpha(\theta))_i$ for all $\theta\in\Theta$, $i\in I$. The model $f$ is \emph{$G$-invariant} if it is $\alpha$-invariant for all $\alpha\in G$.
\end{definition}

\begin{example}
Consider the binomial model $f_i(\theta) = \binom{i}{n}\theta^i(1-\theta)^{n-i}$ on $I=\{0,\dotsc,n\}$. The group $\mathbb Z/2=\{1,\tau\}$ acts on $I$ by $\tau\cdot i = n-i$. The model $f$ is $\mathbb Z/2$-invariant via the function $\overline \tau(\theta)\coloneqq 1-\theta$.
\end{example}

If it exists, the maximum likelihood estimator is always $G$-invariant for any $G$.

We now bring the concept of invariance to Markov combinations. First we need to ensure that the group operation is compatible with the chosen meta-categorisation.

\begin{definition}
Let $p:I\to M$ and $q:J\to M$ be category mappings and let $G$ be a group acting on both $I$ and $J$. The group action is \emph{compatible} with $(p,q)$ if for all $i\in I$ and $j\in J$, if $p(i)=q(j)$ then $p(\alpha\cdot i)=q(\alpha\cdot j)$. 
\end{definition}

\begin{remark}
If the action of $G$ is compatible with $(p,q)$, then there is an induced action of $G$ on $M$ given by $\alpha\cdot p(i)\coloneqq p(\alpha\cdot i).$ This is well-defined since if $p(i)=p(i')$ then $p(\alpha\cdot i)=q(\alpha\cdot j)=p(\alpha\cdot i')$, where $j\in J$ is any element such that $q(j)=p(i)$.
\end{remark}

\begin{definition}
Let $G$ be a group acting on two sets of categories $I$ and $J$. Two statistical models $f:\Theta\to\Delta^I$ and $g:\Theta\to\Delta^J$ are \emph{jointly $G$-invariant} if, for all $\alpha\in G$, both are $\alpha$-invariant via the same injective function $\overline\alpha:\Theta\to\Theta$.
\end{definition}

\begin{theorem}
Let $p: I\to M$ and $q: J\to M$ be category mappings, $G$ a group operating on $I$ and $J$ via a group action compatible with $(p,q)$, and $f:\Theta\to \Delta^I$, $g:\Theta\to\Delta^J$ two $G$-invariant models.
\begin{enumerate}
\item There is an induced action of $G$ on ${I\times_M J}$ given by $\alpha\cdot(i,j)=(\alpha(i),\alpha(j))$.
\item If $f$ and $g$ are jointly $G$-invariant and meta-consistent, then $f\star g$ is $G$-invariant.
\item The lower Markov combination $f\lowerstar g$ is $G$-invariant.
\item The upper Markov combination $f\upperstar g$ is $G$-invariant.
\item The super Markov combination $f\otimes g$ is $G$-invariant. If $h:\Theta\to\Delta^{M}$ is a $G$-invariant model, then the structured super Markov combination $f\otimes_h g$ is $G$-invariant. 
\end{enumerate}
\end{theorem}

\newcommand{\alphabar}{\overline\alpha}
\newcommand{\id}{\operatorname{id}}

\begin{proof}
(1) The given operation is well-defined because the action of $G$ is compatible with $(p,q)$. The axioms of a group action hold, i.e.\ for $\beta,\alpha\in G$, $\mathbf 1\in G$ the neutral element, and $x\in I\times_M J$ we have
\begin{align*}
\mathbf{1}\cdot x &= x,\\
\beta\cdot (\alpha \cdot x) &= (\beta\cdot \alpha)\cdot x,
\end{align*}
because $G$ is a group action on $I$ and $J$.

\medskip
(2) Let $\alpha\in G$ and $f$, $g$ $\alpha$-invariant via the injective map $\overline\alpha:\Theta\to\Theta$. Then, using the induced $G$-action on $M$ we get
\[
f(\overline\alpha\,\theta)_{M,k}=\sum_{p(i)=k}f(\theta)_{\alpha^{-1}(i)} = \sum_{p(i) = \alpha^{-1}(k)} f(\theta)_i = f_{M,\alpha^{-1}(k)}(\theta).
\]
Therefore,
\[
(f\star g)(\overline \alpha\,\theta)_{i,j} =
\frac{f(\overline\alpha\,\theta)_{i}\, g(\overline\alpha\,\theta)_{j}}{f(\overline\alpha\,\theta)_{M,k}} =
\frac{f(\theta)_{\alpha^{-1}(i)}\, g(\theta)_{\alpha^{-1}(j)}}{f(\theta)_{M,\alpha^{-1}(k)}} =
(f\star g)(\theta)_{\alpha^{-1}(i,j)}.
\]

\medskip
(3) Let $\alpha\in G$, $f$ be $\alpha$-invariant via $\overline \alpha_1:\Theta\to\Theta$ and $g$ be $\alpha$-invariant via $\overline\alpha_2:\Theta\to\Theta$. Let $\overline\alpha = (\overline\alpha_1,\overline\alpha_2):\Theta\times\Theta\to\Theta\times\Theta.$ Then $\overline \alpha$ is injective and if $(\theta_1,\theta_2)\in\Theta\times \Theta$ with $f(\theta_1)_{M,k} = g(\theta_2)_{M,k}$ then
\[
f(\overline\alpha_1(\theta_1))_{M,k} = f(\theta_1)_{M,\alpha^{-1}(k)} = g(\theta_1)_{M,\alpha^{-1}(k)}= g(\overline\alpha_2 (\theta_2))_{M,k},
\]
so $\overline\alpha(\theta_1,\theta_2)$ is also a parameter of $f\lowerstar g$. We compute
\[
(f\lowerstar g)(\overline\alpha(\theta_1,\theta_2))_{i,j}
= \frac{f(\theta_1)_{\alpha^{-1}(i)}\, g(\theta_2)_{\alpha^{-1}(j)}}{f(\theta_1)_{M,\alpha^{-1}(k)}}
= (f\star g)(\theta_1,\theta_2)_{\alpha^{-1}(i,j)}.
\]

\medskip
(4) Let $\alpha\in G$ and $f$, $g$ be $\alpha$-invariant via $\overline\alpha_1$, $\overline\alpha_2 : \Theta\to\Theta$, respectively. Define $\overline\alpha=(\alphabar_1,\alphabar_2,\id):\Theta\times\Theta\times\{0,1\}\to\Theta\times\Theta\times\{0,1\}$.
Then, for $\ell\in\{0,1\}$,
\[
(f\upperstar g)(\overline\alpha(\theta_1,\theta_2,\ell))_{i,j}
= (f\upperstar g)(\theta_1,\theta_2,\ell)_{\alpha^{-1}(i,j)}.
\]

\medskip

(5)
Let $\alpha\in G$ and $f$, $g$ be $\alpha$-invariant via $\overline\alpha_1$, $\overline\alpha_3 : \Theta\to\Theta$, respectively. For the Super-Markov combination, define $\overline\alpha:\Theta^3\times\{0,1\}\to \Theta^3\times\{0,1\}$ by
\[
\overline\alpha(\theta_1,\theta_2,\theta_3,\ell)
= \begin{cases}
(\alphabar_1(\theta_1),\alphabar_1(\theta_2),\alphabar_3(\theta_3),\ell) \text{ if } \ell=0,\\
(\alphabar_1(\theta_1),\alphabar_3(\theta_2),\alphabar_3(\theta_3),\ell) \text{ if } \ell=1.
\end{cases}
\]
Then $f\otimes g$ is $\alpha$-invariant via $\overline \alpha$.

If in addition $h$ is $\alpha$-invariant via $\alphabar_2$, then $f\otimes_h g$ is $\alpha$-invariant via
\[
(\alphabar_1,\alphabar_2,\alphabar_3):\Theta^3\to\Theta^3.\qedhere
\]
\end{proof}

\subsection{Maximum likelihood estimation of the Markov combination of saturated models} In general, it is difficult to derive maximum likelihood estimators (MLE) of a Markov combination given MLEs of its components.

For the Markov combination variants where parameters are shared between the components, one runs into the problem of estimating incompatible parameters. For instance, consider the meta-Markov combination of distributions $f:\Theta\to\Delta^I$ and $g:\Theta\to\Delta^I$ where $M=\{1\}$. Then $f\star g$ is the product $f_i(\theta)\,g_j(\theta)$. From MLEs $\varphi:\Delta^{I}\to \Theta$ for $f$ and $\psi:\Delta^J\to\Theta$ for $g$ we could estimate $\theta_1 = \varphi(u)$ and $\theta_2 = \psi(v)$, where 
\[
u_i = \sum_{j\in J} x_{ij} \quad \text{and}\quad
v_j = \sum_{i\in I} x_{ij}. 
\]
However, there does not necessarily exist $\theta_0$ such that $f(\theta_0)g(\theta_0) = f(\theta_1)g(\theta_2)$. 

Even in the Markov combination variants with separate parameters for each component, one runs into problems because of the conditional distributions present. For instance, for the structured Super-Markov
\[
\frac{f_i(\theta_1)}{f_{M,k}(\theta_1)}\cdot h_k(\theta_2)\cdot \frac{g_j(\theta_3)}{g_{M,k}(\theta_3)}.
\]
one could estimate $\theta_1$, $\theta_2$ and $\theta_3$ using MLEs for $f$, $h$, and $g$, respectively. However, an MLE for $f$ is not necessarily an MLE for the conditional distribution $f_i/f_{M,k}$ for all $k$. Rather, one would need estimators $\varphi_k$ for each of the conditional distributions $f_i/f_{M,k}$, which would again  estimate potentially incompatible parameters $\theta_{1,k}$.

It is therefore our impression that MLEs for Markov combinations can only be found on a case-by-case basis. To conclude this subsection, we give such an MLE when $f$ and $g$ are saturated.

\begin{theorem}\label{rational-mle}
Let $I\to M$ and $J\to M$ be category mappings. Let $\mathcal M\subseteq \Delta^{I\times J}$ be the Markov combination of the saturated models $\Delta^I$ and $\Delta^J$ as described in Example~\ref{saturated-consistent-example}. For $(x_{ij})\in\Delta^{I\times_M J}$, let
\begin{align*}
u_i &= \sum_{j\in J} x_{ij}\quad (i\in I),\\
v_j &= \sum_{i\in I} x_{ij}\quad (j\in J),\\
{\hat x}_{ij} &= \frac{u_i v_j}{u_{M,k} \overline x},
\end{align*}
where $\overline x = \sum_{ij} x_{ij}$.
Then $\varphi: x\mapsto \hat x$ is an MLE for $\mathcal M$.
\end{theorem}
\begin{remark}
In our setup where $x$ is a distribution, the term $\overline x$ equals one and can be omitted. But if the $x_{ij}$ are given as intensities instead, then one must divide by the sum of the intensities, which gives the above formula with $\overline x$.
\end{remark}

\begin{proof}[Proof of Theorem~\ref{rational-mle}]
We have $\im(\varphi) = \mathcal M$ and the map $\varphi$ is rational in the entries of $x$. Therefore, we can use a general criterion to show that $\varphi$ is the MLE of the model it parametrizes \cite{duarte-marigliano}. Specifically, we need to exhibit $m\in \mathbb N$, a matrix $H$ of size $m\times |I\times_M J|$, and a vector $\lambda\in\mathbb R^{|I\times_M J|}$ such that the entries of each column of $H$ sum to zero and
\begin{equation}\label{horn-pair}
\varphi(x)_{ij} = \lambda_{ij}
\prod_{\ell=1}^m\left(\sum_{(i',j')\in I\times_M J} h_{\ell,(i',j')}x_{i'j'}\right)^{h_{\ell,(i,j)}}
\end{equation}
for all $x\in\Delta^{|I\times_M J|}$ and $(i,j)\in I\times_M J$. 

We define $\lambda = (1 \cdots 1)$ and $H$ as the direct sum of $|M|$ matrices $H_k$ augmented by a row of $-1$ entries.
For all $k\in M$, let $H_k$ be the matrix of size $(|I|+|J|+|\{k\}|)\times(|I_k \times J_k|)$ defined as follows, where $\delta_{a,b}$ is the Kronecker Delta:
\begin{align*}
(H_k)_{i, (i',j')} &= \delta_{i,i'},\\
(H_k)_{j, (i',j')} &= \delta_{j,j'},\\
(H_k)_{k, (i',j')} &= -1.
\end{align*}
For instance, if $I_k=\{1,2,3\}$, $J_k=\{1,2,3,4\}$, and we order $|I_k\times J_k|$ lexicographically, then
\[
H_k = \begin{pmatrix}
1&1&1&1& & & & & & & & \\
 & & & &1&1&1&1& & & & \\
 & & & & & & & &1&1&1&1\\
1& & & &1& & & &1& & & \\
 &1& & & &1& & & &1& & \\
 & &1& & & &1& & & &1& \\
 & & &1& & & &1& & & &1\\
-1&-1&-1&-1&-1&-1&-1&-1&-1&-1&-1&-1
\end{pmatrix}.
\]
Each column of $H$ splits into a column of $H_k$ (whose entries sum to $1$) for some $k$, and an entry of $-1$ due to the extra row. The sum of the entries of each column of $H$ is therefore zero.

We now analyze Equation~\eqref{horn-pair}. Denote the expression inside the parenthesis as $w_\ell$. Here, $\ell$ can be an element of $I$, $J$, or $M$, or $\ell=0$ to represent the additional row. For $i\in I$ we have
\[
w_i = \sum_{(i',j')\in I\times_M J} h_{i,(i',j')}x_{i',j'}
= \sum_{j'\in J} x_{ij'} = u_i.
\]
Similarly, we have $w_j=v_j$ for all $j\in J$. Because of the block diagonal form of the upper piece of $H$, for $k\in M$ we have
\[
w_k = \sum_{(i',j')\in I\times_M J} h_{k,(i',j')} = \sum_{(i',j')\in I_k\times J_k} -x_{i',j'} = -u_{M,k}.
\]
Finally, the additional row gives $w_0 = -\overline x$.

Analysing Equation~\ref{horn-pair} further, we see that for all $(i,j)\in I\times_M J$, the factor $w_\ell^{h_{\ell,(i,j)}}$ is $\neq 1$ in only four cases. If $\ell=i$ then this factor equals $u_i$; if $\ell=j$ then it equals $v_j$; if $\ell=k$ then it equals $(-u_{M,k})^{-1}$; and if $\ell=0$ then it equals $(-\overline x)^{-1}$. Therefore, the right-hand side of~\eqref{horn-pair} equals
\[
\frac{u_i v_j}{u_{M,k}\overline x},
\]
as required.
\end{proof}

\begin{remark}
The pair $(H,\lambda)$ is called the \emph{Horn pair} associated to $\varphi.$
\end{remark}
\color{black}

\section{Algorithms}\label{sec:algorithms}

We now detail algorithms to compute the distributions of meta-Markov combinations and to sample from these.

\newenvironment{algoblock}
{\setlength{\leftskip}{\parindent}}
{\setlength{\leftskip}{\parindent}}

\begin{algorithm}\label{algo:aggregates} Aggregates.
\smallskip

\begin{algoblock}
\noindent
\textbf{Input:} a parameter space $\Theta$, a category mapping $p:I\to M$, and a parametric model $f(\theta)=(f_i(\theta))_{i\in I}.$
\smallskip

\noindent
\textbf{Output:} the aggregate $f_M$ and the aggregate categories $I_k$, $k\in M$.

\noindent
\begin{enumerate}
\item for all $k\in M,$ set\\
$I_k \leftarrow \{i\in I\mid p(i)=k\}$,\\
$f_{M,k}(\theta) \leftarrow \sum_{i\in I_k}f_i(\theta)$.
\item \textbf{return} $f_{M}$, $(I_k)_{k\in M}$
\end{enumerate}
\end{algoblock}
\end{algorithm}

\color{black}

\begin{algorithm}\label{algo:meta-markov} Meta-Markov combinations.
\smallskip

\begin{algoblock}
\noindent
\textbf{Input:} a parameter space $\Theta$, category mappings $p:I\to M$ and $q:J\to M$, and parametric models $f(\theta)=(f_i(\theta))_{i\in I}$ and $g(\theta)=(g_j(\theta))_{j\in J}$.
\smallskip

\noindent
\textbf{Output:} the meta-Markov combination $f\star g$ if $f,g$ meta-consistent, an error otherwise.
\smallskip

\noindent
\begin{enumerate}
\item Compute $f_M$, $g_M$, $I_k$, $J_k$ $(k\in M)$ with Algorithm~\ref{algo:aggregates}.
\item for all $k\in M$,
if $f_{M,k}(\theta)\neq g_{M,k}(\theta)$ \textbf{return error}.
\color{black}
\item Set $I\times_M J \leftarrow \{(i,j,k) \mid k\in M, i\in I_k, j\in J_k\}$.
\item for $(i,j,k)\in I\times_M J$, set
\[
(f\star g)_{i,j,k}(\theta) \leftarrow \frac{f_i(\theta)g_j(\theta)}{f_{M,k}(\theta)}.
\]
\item \textbf{return} $(f\star g)(\theta)$
\end{enumerate}
\end{algoblock}

\end{algorithm}

\begin{algorithm}\label{algo:lower-restr} Restricted lower Markov combinations.
\smallskip

\begin{algoblock}
\noindent
\textbf{Input:} as in Algorithm~\ref{algo:meta-markov}.
\smallskip

\noindent
\textbf{Output:} the restricted lower Markov combination $(f\lowerstar g)(\theta)$.
\smallskip

\noindent
\begin{enumerate}
\item as in Algorithm~\ref{algo:meta-markov}.
\item solve the following system of equations for $\theta$:
\[
f_{M,k}(\theta) = g_{M,k}(\theta) \text{ for all } k\in M\qquad (\ast).
\]
\item set $\Theta'\leftarrow\{\theta\in\Theta\mid (\ast)\}$.
\item \textbf{return} the output of Algorithm~\ref{algo:meta-markov} called with the parameter space $\Theta'$.
\end{enumerate}
\end{algoblock}

\end{algorithm}

\color{black}

\begin{algorithm}\label{algo:sampling} Sampling from a Meta-Markov combination.
\smallskip

\begin{algoblock}
\noindent
\textbf{Input:} as in Algorithm~\ref{algo:meta-markov}, assuming $f$ and $g$ are meta-consistent; $\theta\in\Theta$.
\smallskip

\noindent
\textbf{Output:} A sample $(i,j)$ of $(f\star g)(\theta)$.
\smallskip

\noindent
\begin{enumerate}
\item Sample $i\in I$ with probability $f_i(\theta)$.
\item Set $J_{p(i)}\leftarrow \{j\in J\mid q(j) = p(i)\}$.
\item Sample $j\in J_{p(i)}$ with probability $g_j(\theta)/g_{M,p(i)}(\theta)$.
\item \textbf{return} $(i,j)$.
\end{enumerate}
\end{algoblock}
\end{algorithm}

\begin{algorithm} Sampling from a structured Super-Markov combination.
\smallskip

\begin{algoblock}
\noindent
\textbf{Input:} as in Algorithm~\ref{algo:meta-markov}, with the addition of a parametric model $h(\theta)=(h_i(\theta))_{k\in M}.$ and a parameter $\theta\in\Theta$.
\smallskip

\noindent
\textbf{Output:} A sample $(i,j)$ of $(f\otimes_h g)(\theta)$.
\smallskip

\noindent
\begin{enumerate}
\item Sample $k\in M$ with probability $h_k$.
\item Set $I_k\leftarrow\{i\in I\mid p(i)=k\}$.
\item Sample $i\in I_k$ with probability $f_i(\theta)/f_{M_k}(\theta)$.
\item Continue with steps (2)--(4) of Algorithm~\ref{algo:sampling}.
\end{enumerate}
\end{algoblock}
\end{algorithm}

\section{Discussion}\label{sec:discussion}

Markov combinations provide a general way to combine discrete statistical models. The different available variants, and the choice of meta-categorisation, provide a flexible way to adapt Markov combinations to practicioners' needs. This flexibility may come at a cost, however: the meta-categorisation needs to be carefully chosen, and to fit the problem at hand.

A way to gain intuition on which meta-categorisations are useful would be to examine Markov combinations in more classes of models. We see the possibility of bringing our methods to classes such as relational models~\cite{relational-models} and undirected probabilistic graphical models for discrete variables. The case of DAG graphical models, i.e.\ Bayesian networks, is already covered by staged trees. We reserve these considerations to future work.

Other possible theoretical considerations for future work are the study of moments of discrete Markov combinations and the estimation thereof, a study of partial likelihood estimation or profile likelihood estimation, a study of maximum likelihood estimators for more types of Markov-combined models, and the extension of Markov combinations to discrete but not necessarily finitely-supported models.

\section*{Acknowledgements}
ER and OM were supported by the European Research Executive Agency (Project no.\ 101061315–MIAS–HORIZON-MSCA-2021-PF-01). 
ER acknowledges the financial support from the “Hub Life Science - Digital Health (LSHDH) PNC-E3- 2022-23683267 - Progetto DHEAL-COM ”, funded by the Italian Ministry of Health within the Piano Nazionale Complementare for the “PNRR Ecosistema Innovativo della Salute”.  
This work was also partially supported by the MIUR Excellence Department Project awarded to the Department of Mathematics, University of Genoa.

\bibliographystyle{plain}
\bibliography{refs}

\end{document}